\documentclass[a4paper,12pt]{article}
\usepackage{amsmath,amsthm,amssymb,amsfonts}
\usepackage{bbm,bm}
\usepackage{latexsym}
\usepackage{mathrsfs}
\usepackage{threeparttable}
\usepackage{tabularx}
\usepackage{booktabs}
\usepackage{graphicx,subfigure}
\usepackage{color}
\usepackage{indentfirst}
\usepackage{geometry}
\usepackage{caption}
\usepackage{lineno}
\usepackage{abstract}
\usepackage{enumerate,mdwlist}
\usepackage[numbers,sort&compress]{natbib}
\usepackage[colorlinks,linkcolor=blue,citecolor=red]{hyperref}
\usepackage{epstopdf}

\def \[{\begin{equation}}
\def \]{\end{equation}}

\newtheorem{thm}{Theorem}[section]

\newtheorem{lem}[thm]{Lemma}
\newtheorem{cor}[thm]{Corollary}

\newtheorem{conj}[thm]{Conjecture}

\begin{document}

\setlength{\baselineskip}{20pt}
\begin{center}{\Large \bf Cubic vertices of minimal bicritical graphs}\footnote{This work
is supported by NSFC\,(Grant No. 12271229).}

\vspace{4mm}

{Jing Guo, Hailun Wu, Heping Zhang \footnote{The corresponding authors.
E-mail address: guoj20@lzu.edu.cn (J. Guo), wuhl18@lzu.edu.cn (H. Wu), zhanghp@lzu.edu.cn (H. Zhang).}
\renewcommand\thefootnote{}\footnote{}}

\vspace{2mm}

\footnotesize{ School of Mathematics and Statistics, Lanzhou University, Lanzhou, Gansu 730000, P. R. China}

\end{center}
\noindent {\bf Abstract}:
A graph $G$ with four or more vertices is called bicritical if the removal of
any pair of distinct vertices of $G$ results in a graph with a perfect matching.
A bicritical graph is minimal if
the deletion of each edge results in a non-bicritical graph.
Recently, Y. Zhang et al. and F. Lin et al. respectively showed that bicritical graphs
without removable edges and minimal bricks have at least four cubic vertices.
In this note, we show that minimal bicritical graphs also have  at least four cubic vertices,
so confirming O. Favaron and M. Shi's conjecture
in the case of $k=2$ on minimal $k$-factor-critical graphs.

\vspace{2mm} \noindent{\bf Keywords}: Minimal bicritical graph; Perfect matching;
Cubic degree; Cubic vertex
\vspace{2mm}

\noindent{AMS subject classification:} 05C70,\ 05C07

 {\setcounter{section}{0}
\section{Introduction}\setcounter{equation}{0}

Graphs considered in this paper are finite and simple.
Let $G$ be a graph with vertex set $V(G)$ and edge set $E(G)$.
The {\em order} of graph $G$ is the cardinality of $V(G)$.
For a vertex $x$ in $G$, the degree of $x$, denoted by $d_{G}(x)$, is the number of edges  incident with $x$.
If $d_{G}(x)=3$, then $x$ is called a {\em cubic vertex}.
A {\em perfect matching} $M$ of $G$ is a set of  edges such that each vertex is  incident with exactly one edge
 of $M$. A connected graph $G$ with at least two vertices is {\em matching covered}
if each edge lies in a perfect matching of $G$.
An edge $e$ of a matching covered graph $G$ is {\em removable} if $G-e$ is also matching covered.
A graph $G$ with four or more vertices is called {\em bicritical} \cite{LL}
if the removal of any pair of distinct vertices of $G$ results in a graph with a perfect matching.
We can see that each bicritical graph is matching covered.

A bicritical graph is {\em minimal} if for every edge,
the deletion of it results in a graph that is not bicritical.
For convenience, we call an edge $e$ of a bicritical graph $G$ is {\em deletable}
if $G-e$ is also bicritical. So a minimal bicritical graph has no deletable edge.
Obviously,  a bicritical graph without removable edges is minimal bicritical.
L. Lov\'{a}sz and M. D. Plummer \cite{LP75} proved that
every minimal bicritical graph does not contain $K_{3,3}$ as a subgraph and
contains no wheel as a subgraph if it is not a wheel itself.
For details concerning minimal bicritical graphs,
the reader may refer to \cite{LP} by L. Lov\'{a}sz and M. D. Plummer.

A 3-connected and  bicritical graph is called a {\em brick},
which plays a key role in matching theory of graphs.
J. Edmonds et al. \cite{ELW} and L. Lov\'{a}sz \cite{LO} proposed and developed the
``tight cut decomposition" of matching covered graphs into list of bricks and braces
in an essentially unique manner.
A brick $G$ is {\em minimal} if $G-e$ is not a brick for any $e\in E(G)$.
M. H. de Carvalho et al. \cite{CLM} proved that every minimal brick contains a cubic vertex.
This was extended by S. Norine and R. Thomas \cite{NT} as follows.

\begin{thm}[\cite{NT}]\label{NT}
Every minimal brick has at least three cubic vertices.
\end{thm}

Further, F. Lin et al. \cite{LZL} showed the existence of four such cubic vertices.

\begin{thm}[\cite{LZL}]\label{LZL}
Every minimal brick has at least four cubic vertices.
\end{thm}

Generally, a graph $G$ of order $n$ is said to be {\em $k$-factor-critical} \cite{F, Y}
for integers $1\leq k < n$, if the removal of any $k$ vertices results in a graph
with a perfect matching. 1- and 2-factor-critical graphs are
the well-known factor-critical and bicritical graphs, respectively.
For more about the $k$-factor-critical graphs, the reader may refer to a monograph \cite{YL}.
In particular, a graph $G$ is called {\em minimal $k$-factor-critical}
if $G$ is $k$-factor-critical but $G-e$ is not $k$-factor-critical for any $e\in E(G)$.
Since a $k$-factor-critical graph is $k$-connected and $(k+1)$-edge-connected
and thus has minimum degree at least $k+1$, O. Favaron and M. Shi \cite{FS} proposed
the following conjecture and confirmed it for $k=1, n-6, n-4$ and $n-2$.

\begin{conj}[\cite{FS}]\label{conj1}
Let $G$ be a minimal $k$-factor-critical graph of order $n$ with $1\leq k<n$.
Then $G$ has minimum degree $k+1$.
\end{conj}

In previous papers \cite{GZ1, GZ2}, we had confirmed Conjecture \ref{conj1}
for $k=n-8$ and $n-10$. For $k=2$, only partial solutions can be obtained.
For a 3-connected minimal bicritical graph $G$, it is also a minimal brick
since $G-e$ is not bicritical for each $e\in E(G)$, yielding $G-e$ is not a brick.
So Theorem \ref{LZL} implies that a 3-connected minimal bicritical graph has
at least four cubic vertices. Further,  as mentioned above,
bicritical graphs without removable edges are minimal bicitical.
Y. Zhang et al. \cite{ZWY} obtained the following theorem.

\begin{thm}[\cite{ZWY}]\label{ZWY}
Every bicritical graph  without removable edges has at least four cubic vertices.
\end{thm}

In addition to these two classes of graphs,
there are other minimal bicritical graphs.
 For example, Fig. \ref{tu-1} presents a minimal bicritical graph $G$, which has
 a removable edge $e$, but is not a minimal brick.
By using brick decomposition described in \cite{LO} and Theorem \ref{LZL},
generally, we confirm Conjecture \ref{conj1} to be true for $k=2$.
The following is our main theorem.

\begin{thm}(Main Theorem)\label{main}
Every minimal bicritical graph
has at least four cubic vertices.
\end{thm}

\section{Some preliminaries}

In this section, we give some graph-theoretical terminologies and notations,
and some preliminary results as well.

For a set $S\subseteq V(G)$, let $G[S]$ denote the subgraph of $G$
induced by $S$ in $G$, and $G-S=G[V(G)-S]$.
For an edge $uv\in E(G)$, $G-uv$ stands for the graph $(V(G), E(G)\setminus\{uv\})$.
Similarly, if $u$ and $v$ are two non-adjacent vertices of $G$,
then $G+uv$ stands for the graph $(V(G), E(G)\cup \{uv\})$.
The union of graphs $G$ and $H$, written $G\cup H$, is the graph
with vertex set $V(G)\cup V(H)$ and edge set $E(G)\cup E(H)$.

An {\em independent set} of a graph is a set of pairwise nonadjacent vertices.
Let $G$ be a connected graph.
A {\em $k$-vertex cut} of a graph $G$ is a set $S \subset V(G)$ with $|S|=k$
such that $G-S$ is disconnected.
A graph $G$ is called {\em $k$-connected} for non-negative integer $k$
if $|V(G)|>k$ and $G-X$ is connected for every set $X \subset V(G)$ with $|X|<k$.

Let $G$ be a graph with a perfect matching. A {\em barrier} of graph $G$ is
a subset $B$ of $V(G)$ for which $c_{o}(G-B)=|B|$,
where $c_{o}(G-B)$ denotes the number of odd components of $G-B$.
A 2-vertex cut of graph $G$ that is not a barrier
is referred to as a {\em 2-separation} of $G$.
So every 2-vertex cut of a bicritical graph is a 2-separation as
each of its barriers is a singleton.

For bicritical graphs, L. Lov\'{a}sz and M. D. Plummer obtained actually the following decomposition theorem contained in the proof of Lemma 5.2.8 in \cite{LP}.

\begin{thm}[\cite{LP}] \label{bicritical}
If $G=G_{1}'\cup G_{2}'$, $V(G_{1}') \cap V(G_{2}')=\{u, v\}$,
where $G_{1}'$ and $G_{2}'$ are connected graphs different from $K_{2}$,
and $G_{i}=G_{i}'+uv$ when $uv\notin E(G_{i}')$, otherwise
$G_{i}=G_{i}'$ for $i=1, 2$,
then $G$ is bicritical if and only if both $G_{1}$ and $G_{2}$ are bicritical.
\end{thm}

In order to prove our main result,
we first recall brick decomposition on bicritical graphs with a 2-separation
 defined in \cite{LO}.

Let $G$ be a graph with a 2-vertex cut $\{u, v\}$. We can write $G$ as the union of
two edge-disjoint subgraphs $G_{1}'$ and $G_{2}'$, where $G_{1}'$ and $G_{2}'$ have
more than two vertices and $V(G_{1}') \cap V(G_{2}')=\{u, v\}$.
Add an edge between $u$ and $v$ to $G_{i}'$ to get a graph $G_{i}$, $i=1, 2$,
provided $G_i'$ has no such edge. If $uv\in E(G_{i}')$, we let $G_{i}=G_{i}'$.
We say that $G$ performs one decomposition and
refer to $uv$ as a {\em marker edge} of $G_{i}$ for $i=1, 2$. For example, see Fig. \ref{tu-1}.

When $G$ performs one decomposition, each $G_{i}$ is a simple graph
different from $K_{2}$ and has exactly one marker edge $uv$.
%A proof of Theorem \ref{bicritical}
%(see \cite{LP} shows that each $G_{i}$ is bicritical if $G$ is bicritical).
If some $G_{i}$ is not 3-connected, then $G_{i}$ has a 2-vertex cut,
say $\{u_{i}, v_{i}\}$. By the same manner as above,
$G_{i}$ can be decomposed into two graphs
$G_{i1}$ and $G_{i2}$, which are different from $K_{2}$.
At this time, we say that $G$ performs two decompositions.
If $\{u_{i}, v_{i}\}=\{u, v\}$, then $uv$ is a marker edge of $G_{i1}$ and $G_{i2}$
since $uv$ is a marker edge of $G_{i}$.
Otherwise, $uv$ is a marker edge of exactly one of both graphs $G_{i1}$ and $G_{i2}$.
This decomposition procedure can be repeated
until a list of 3-connected graphs is obtained.

In particular, if $G$ is a bicritical graph but not a brick, then $G$ has a 2-separation.
By Theorem \ref{bicritical}, both $G_{1}$ and $G_{2}$ are bicritical.
We can continue above procedure on each bicritical graph until a list of bricks is obtained.
Such a decomposition procedure is called a {\em brick decomposition} of $G$.
However, L. Lov\'{a}sz \cite{LO} showed that the final list of bricks does not depend on
which 2-separation is chosen at each step in the brick decomposition procedure.

\begin{thm}[\cite{LO}] \label{LO1}
%The list of bricks obtained in a brick decomposition procedure is uniquely determined.
Any two brick decompositions of a bicritical graph yield the same list of bricks.
\end{thm}

We call these bricks produced during the brick decomposition procedure are
the bricks of $G$ and denote by $b(G)$ the number of these bricks.
Let $s(G)$ denote the times of decompositions in the brick decomposition procedure of $G$.

%Let $G$ be a bicritical graph.
Next we use a {\em binary tree}, denoted by $T_{G}$,
to describe a brick decomposition of $G$:
The vertices of $T_{G}$ are the bicritical graphs which arise during
the entire procedure of the brick decomposition of $G$.
The root of $T_{G}$ is $G$ and the leaves of $T_{G}$ are the bricks of $G$.
If a bicritical graph $G_i$ is decomposed into two bicritical graphs $G_{i1}$ and $G_{i2}$,
then we give two edges from $G_i$ to $G_{i1}$ and $G_{i2}$.

So $s(G)$ and  $b(G)$ are equal to the numbers of non-leaves and leaves of $T_G$, respectively.
Then $|V(T_{G})|=b(G)+s(G)$ and $|E(T_G)|=2s(G)$.
Since $T_G$ is a tree,  $|E(T_{G})|=$$|V(T_{G})|$ $-1$, which implies that
\begin{equation}\label{eq}
s(G)=b(G)-1.
\end{equation}

%Since every vertex of $T_{G}$ other than the root and leaves has degree three and
%the root of $T_{G}$ has degree two, we have
%
%
%\begin{center}
%$\sum\limits_{v \in V(T_{G})} d_{T_{G}}(v)=2|E(T_{G})|=2(b(G)+s(G)-1)=b(G)+3(s(G)-1)+2$.
%\end{center}
%So $s(G)=b(G)-1$.

If each 2-separation of $G$ is used exactly once
in the brick decomposition procedure of $G$,
then the number of marker edges of the bricks of $G$ is exactly $2s(G)$.
Otherwise, the number of marker edges of the bricks of $G$ is less than $2s(G)$.
Thus, we have that after a brick decomposition procedure of $G$,
the number of marker edges of the bricks of $G$ is at most $2s(G)$.

\begin{lem} \label{marker}
Let $G$ be a bicritical graph but not a brick.
Then after a brick decomposition procedure of $G$,
$G$ has at least two bricks, each of which has exactly one marker edge.
\end{lem}

\begin{proof}
Let $G_{1}, G_{2}, \ldots , G_{k}$, $k\geq 2$, be the bricks of $G$
after a brick decomposition of $G$.
Since $G$ is not a brick, for each $1 \leq i \leq k$, $G_{i}$ has at least one marker edge.
Suppose to the contrary that there is at most one brick of $G$ with exactly one
marker edge. Then the number of marker edges of the bricks of $G$ is at least
$2(b(G)-1)+1=2s(G)+1$ by taking Eq. (\ref{eq}) into account, a contradiction.
\end{proof}

\section{Proof of Main Theorem}

In this section, we will give a detailed proof of  Theorem \ref{main}.
At first, we introduce some properties of a bicritical graph with a 2-separation.

%Let $G_{1}, G_{2}, \ldots , G_{k}$ be pairwise disjoint connected graphs
%different from $K_{2}$, where $k\geq 2$.
%Let $e_{i}$ be an edge of $G_{i}$ with ends $u_{i}$ and $v_{i}$, $1\leq i\leq k$.
%We denote by $(G_{1} \oplus G_{2} \oplus \cdots \oplus G_{k})_{uv}$ the
%graph obtained from $G_{1}, G_{2}, \ldots , G_{k}$ by identifying these vertices
%$u_{1}, u_{2}, \ldots , u_{k}$ to a new vertex $u$ and these vertices
%$v_{1}, v_{2}, \ldots , v_{k}$ to a new vertex $v$,
%and then, deleting all the edges between $u$ and $v$.

%If $G$ is a bicritical graph with a 2-separation, say $\{u, v\}$.
%Let $H_{1}, H_{2}, \ldots ,$ $H_{k}$ be the components of $G-\{u, v\}$, $k\geq 2$.
%If $uv\notin E(G)$, then let $G_{i}'=G[V(H_{i})\cup \{u, v\}]$ and $G_{i}=G_{i}'+uv$,
%$1 \leq i \leq k$. Thus $G=(G_{1} \oplus G_{2} \oplus \cdots \oplus G_{k})_{uv}$.

%Y. Zhang et al. \cite{ZWY} obtained some properties of a bicritical graph with a 2-separation.

\begin{lem}[\cite{ZWY}] \label{ZWY1}
Let $G$ be a bicritical graph with a 2-separation $\{u, v\}$. Then

(1) both $u$ and $v$ have at least two neighbours in each component of $G-u-v$;

(2) if $uv\in E(G)$, then $G-uv$ is bicritical and $uv$ is removable in $G$.
\end{lem}

From Lemma \ref{ZWY1}(2), we can see that if $G$ is a minimal bicritical graph,
then every 2-separation of $G$ is an independent set.
Let $DE(G)$ denote the set of all the deletable edges of a bicritical graph $G$.
The following lemma is important in the proof of main theorem.

\begin{lem}\label{deletable}
Let $G$ be a bicritical graph with a 2-separation $\{u, v\}$ and
$G_{i}'$, $G_{i}$, $i=1, 2$, be as stated of Theorem \ref{bicritical}. Then

(1) if $uv\notin E(G)$, then $DE(G_{i})\backslash\{uv\}=$ $DE(G)\cap E(G_{i})$ for $i=1, 2$;

(2) if $uv\in E(G)$, then $DE(G_{i})\cup \{uv\}=$ $DE(G)\cap E(G_{i})$ for $i=1, 2$.

\end{lem}

\begin{proof}
(1) By Theorem \ref{bicritical}, both $G_{1}$ and $G_{2}$ are bicritical.
We need only to show that $DE(G_{1})\backslash\{uv\}=$ $DE(G)\cap E(G_{1})$ by symmetry.

Suppose that $e\in DE(G)\cap E(G_{1})$. Then $e\in E(G)$ and $G-e$ is bicritical.
Since $uv\notin E(G)$, $e\neq uv$.
For any $u_{1}, v_{1}\in V(G_{1})$, $u_{1}, v_{1}\in V(G_{1}')\subseteq V(G)$.
Let $M$ be a perfect matching of $G-e-\{u_{1}, v_{1}\}$.
If $\{u_{1}, v_{1}\}\cap \{u, v\}=\emptyset$,
since every component of $G-\{u, v\}$ is even,
both $u$ and $v$ are matched to some $G_{i}'$.

If both $u$ and $v$ are matched to $G_{1}'$ under $M$ or
$\{u_{1}, v_{1}\}\cap \{u, v\}\neq \varnothing$, then
$M\cap E(G_{1})$ is a perfect matching of $G_{1}-e-\{u_{1}, v_{1}\}$.

If both $u$ and $v$ are matched to $G_{2}'$ under $M$, then
$(M\cap E(G_{1}))\cup \{uv\}$ is a perfect matching of $G_{1}-e-\{u_{1}, v_{1}\}$.

Thus $G_{1}-e$ is bicritical, which implies that $e\in DE(G_{1})\backslash\{uv\}$.
Then $DE(G)\cap E(G_{1})$ $\subseteq DE(G_{1})\backslash\{uv\}$.

Conversely, suppose that $f\in DE(G_{1})\backslash\{uv\}$.
Then $f\in E(G_{1}')\subseteq E(G)$ and $G_{1}-f$ is bicritical.

For any $u_{1}, v_{1}\in V(G_{1}'-f)=V(G_{1}-f)$,
let $M_{1}$ be a perfect matching of $G_{1}-f-\{u_{1}, v_{1}\}$.
If $uv \notin M_{1}$, then $M_{1}\cup M_{2}$ is
a perfect matching of $G-f-\{u_{1}, v_{1}\}$,
where $M_{2}$ is a perfect matching of $G_{2}-\{u, v\}$.
If $uv \in M_{1}$, then we may consider $G_{2}$. By Lemma \ref{ZWY1} (1),
both $u$ and $v$ are adjacent to at least two vertices in $G_{2}-\{u, v\}$.
Since $G_{2}$ is also matching covered,
$G_{2}$ has a perfect matching $M_{2}'$ missing $uv$.
Then $(M_{1}\setminus \{uv\})\cup M_{2}'$
is a perfect matching of $G-f-\{u_{1}, v_{1}\}$.

Take any $u_{2}, v_{2}\in V(G_{2}')=V(G_{2})$.
Then $G_{2}-\{u_{2}, v_{2}\}$ has a perfect matching $M_{2}$.
If $uv \notin M_{2}$, then $M_{1}\cup M_{2}$
is a perfect matching of $G-f-\{u_{2}, v_{2}\}$,
where $M_{1}$ is a perfect matching of $G_{1}-f-\{u, v\}$.
If $uv \in M_{2}$, then we may consider $G_{1}-f$.
By Lemma \ref{ZWY1} (1), both $u$ and $v$ are adjacent to at least one vertices
in $G_{1}-f-\{u, v\}$. Since $G_{1}-f$ is matching covered,
$G_{1}-f$ has a perfect matching $M_{1}'$ missing $uv$.
Then $M_{1}' \cup (M_{2} \setminus \{uv\})$
is a perfect matching of $G-f-\{u_{2}, v_{2}\}$.

Let $u_{1}\in V(G_{1}')-\{u, v\}$ and $u_{2}\in V(G_{2}')-\{u, v\}$.
Then $u_{1}\in V(G_{1})-\{u, v\}$ and $u_{2}\in V(G_{2})-\{u, v\}$.
Assume that $G_{1}-f-\{u_{1}, u\}$ has a perfect matching $M_{1}$ saturating $v$ and
$G_{2}-\{u_{2}, v\}$ has a perfect matching $M_{2}$ saturating $u$.
Then $M_{1} \cup M_{2}$ is a perfect matching of $G-f-\{u_{1}, u_{2}\}$.

It follows that $G-f$ is bicritical. So $f\in DE(G)$. Then $f\in DE(G)\cap E(G_{1})$.
Thus $DE(G_{1})\backslash\{uv\}\subseteq$$DE(G)\cap E(G_{1})$.
Consequently, $DE(G_{1})\backslash\{uv\}=$ $DE(G)\cap E(G_{1})$.

\smallskip

(2) By Lemma \ref{ZWY1}, $G-uv$ is bicritical. Then $uv\in DE(G)$.
By (1), $DE(G_{i})\backslash\{uv\}=$ $DE(G-uv)\cap E(G_{i})$, $i=1, 2$.
We need only to show that $DE(G-uv)=DE(G)\setminus\{uv\}$.

Obviously, $DE(G-uv)$ $\subseteq DE(G)\setminus\{uv\}$.
Suppose that $h \in DE(G)\setminus\{uv\}$. Then $h\neq uv$ and $G-h$ is bicritical.
Without loss of generality, assume that $h\in E(G_{1}')$. So $h\in E(G_{1})$.
For any $u_{1}, v_{1}\in V(G_{1})$, $u_{1}, v_{1}\in V(G_{1}')$.
Let $M'$ be a perfect matching of $G-h-\{u_{1}, v_{1}\}$.

If both $u$ and $v$ are matched to $G_{1}'$ under $M'$ or
$\{u_{1}, v_{1}\}\cap \{u, v\}\neq \varnothing$ or $uv\in M'$, then
$M'\cap E(G_{1})$ is a perfect matching of $G_{1}-h-\{u_{1}, v_{1}\}$.

If both $u$ and $v$ are matched to $G_{2}'$ under $M'$, then
$(M'\cap E(G_{1}))\cup \{uv\}$ is a perfect matching of $G_{1}-h-\{u_{1}, v_{1}\}$.

Thus $G_{1}-h$ is bicritical. So $h\in DE(G_{1})\backslash \{uv\}$.
Since $DE(G_{1})\backslash\{uv\}=$ $DE(G-uv)\cap E(G_{1})$, $h\in DE(G-uv)$.
Then $DE(G)\setminus\{uv\}\subseteq DE(G-uv)$.
Therefore, $DE(G-uv)=DE(G)\setminus\{uv\}$.
\end{proof}

From Lemma \ref{deletable}, we obtain the following corollary,
which is a version of minimal bicritical graphs of Theorem \ref{bicritical}.

\begin{cor} \label{minimal}
Let $G$ be a graph with a 2-vertex cut $\{u, v\}$ and
$G_{i}'$, $G_{i}$, $i=1, 2$, be as stated of Theorem \ref{bicritical}. Then

(1) if $uv\notin E(G)$ and each $G_{i}$ is minimal bicritical, then $G$ is minimal bicritical;

(2) if $G$ is minimal bicritical, then every edge of $G_{i}$ is not deletable except for $uv$.

\end{cor}

\begin{proof}
(1) By Theorem \ref{bicritical}, $G$ is bicritical.
It suffices to prove that $G$ has no deletable edge.
Since each $G_{i}$ is minimal bicritical, $DE(G_{i})=\emptyset$, $i=1, 2$.
By Lemma \ref{deletable} (1), for each $i=1, 2$,
$DE(G_{i})\backslash\{uv\}$ $=DE(G)\cap E(G_{i})=\emptyset$.
So $DE(G)=\emptyset$. Then $G$ is minimal bicritical.

(2) By Theorem \ref{bicritical}, both $G_{1}$ and $G_{2}$ are bicritical.
Since $G$ is minimal bicritical, $uv\notin E(G)$ and $DE(G)=\emptyset$.
By Lemma \ref{deletable} (1),
$DE(G)\cap E(G_{i})=DE(G_{i})\backslash\{uv\}=\emptyset$, $i=1, 2$.
That is, each $G_{i}-uv$ has no deletable edge.
So the only possible deletable edge in $G_{i}$ is $uv$.
\end{proof}

Note that if $G$ is a minimal bicritical graph with a 2-separation $\{u, v\}$,
then the edge $uv$ may be deletable in some $G_{i}$, $i=1, 2$.
For example, a minimal bicritical graph $G$ with a 2-separation $\{u,v\}$
shown in Fig. \ref{tu-1} can be decomposed into two graphs $G_1$ and $G_2$,
where $uv$ is deletable in $G_{1}$ but not deletable in $G_{2}$.
That is, $G_2$ is  minimal bicritical but $G_1$ is not.
(The bold edge is a marker edge of $G_{1}$ and $G_{2}$.)

\begin{figure}[h]
\centering
\includegraphics[height=3.5cm,width=11cm]{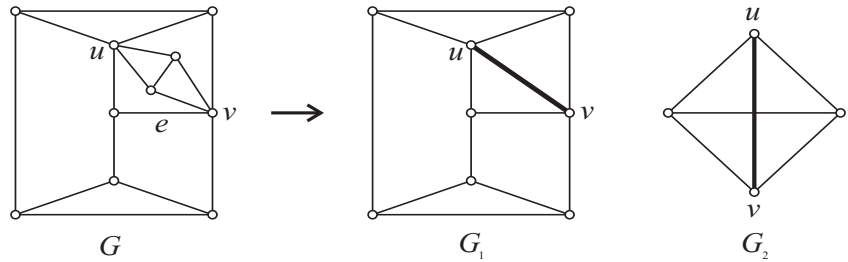}
\caption{\label{tu-1} A minimal bicritical graph $G$ with a removable edge $e$,
but $uv$ is deletable in $G_{1}$.}
\end{figure}

We are now ready to prove our main theorem.

\noindent{\bf Proof of Theorem \ref{main}.}
If $G$ is 3-connected, then $G$ is a minimal brick as mentioned in Section 1.
By Theorem \ref{LZL}, the result holds.

Otherwise, $G$ is 2-connected and has a 2-vertex cut, which must be a 2-separation.
By Lemma \ref{ZWY1}, every 2-separation is an independent set of $G$.
After a brick decomposition procedure of $G$, we obtain a list of bricks of $G$.
By Lemma \ref{marker}, there are at least two bricks of $G$, say $G_{1}$ and $G_{2}$,
which have exactly one marker edge $e_{1}$ and $e_{2}$, respectively.
Let $e_{i}=u_{i}v_{i}$ for $i=1, 2$.
Clearly, both $G_{1}$ and $G_{2}$ are 3-connected.
Since only $u_{i}v_{i}$ in $G_{i}$
is the new edge added during the brick decomposition procedure of $G$,
$E(G_{i})\setminus \{u_{i}v_{i}\}=E(G_{i})\cap E(G)$.
Thus the degree of every vertex of $G_{i}-\{u_{i}, v_{i}\}$
is the same as the degree of a vertex of $G$.

{\bf Claim.} Every edge of $G_{i}-u_{i}v_{i}$, $i=1, 2$,
is not deletable in $G_{i}$.

Since $G$ is minimal bicritical, $DE(G)=\emptyset$.
When $G$ performs the first decomposition, we obtain two bicritical graphs
with exactly one marker edge, respectively.
By Lemma \ref{deletable} (1), every edge of these two graphs is not deletable except for
the two marker edges. If one of the two graphs is not a brick,
then we continue with the decomposition.
By repeatedly applying Lemma \ref{deletable} to each decomposition,
every edge of the decomposed bicritical graphs is not deletable except for their marker edges.
Thus, after the brick decomposition of $G$,
the only possible deletable edges of the final list of bricks are their marker edges.
So Claim holds.

Now we consider whether $G_{i}$, $i=1, 2$, is a minimal brick or not.

If $G_{i}$ is a minimal brick, then $G_{i}$ has at least four cubic vertices by Theorem \ref{LZL}.

Otherwise, there is an edge of $G_{i}$ such that
the deletion of it results in a graph that is also a brick.
Since every edge of $G_{i}$ is not deletable except for $u_{i}v_{i}$ by Claim,
$G_{i}$ only has an edge $u_{i}v_{i}$ such that $G_{i}-u_{i}v_{i}$ is a brick.
If $G_{i}-u_{i}v_{i}$ has a deletable edge, say $e'$,
then $G_{i}-u_{i}v_{i}-e'$ is bicritical.
Note that $G_{i}-u_{i}v_{i}-e'$ is 3-connected.
Otherwise, $G_{i}-u_{i}v_{i}$ is a minimal brick. We are done.
Thus $G_{i}-u_{i}v_{i}-e'$ is a brick. So $G_{i}-e'$ is a brick.
It follows that $e'$ different from $u_{i}v_{i}$ is deletable
in $G_{i}$, which is a contradiction to Claim.
This contradiction implies that $G_{i}-u_{i}v_{i}$ is also a minimal brick.
So $G_{i}-u_{i}v_{i}$ has at least four cubic vertices by Theorem \ref{LZL}.

In either case, each $G_{i}$ has at most two cubic vertices incident with
its marker edge. Then each $G_{i}$ has at least two cubic vertices, which are
the cubic vertices of $G$. Therefore, $G$ has at least four cubic vertices.
~~~~~~~~~~~~~~~~~~~~~~~~~~~~~~~~~~~~~~~~~~~~~~~
~~~~~~~~~~~~~~~~~~~~~~~~~~~~~~~~~~~~~~~~~~~~~~~$\square$

\bigskip

This bound of Theorem \ref{main} is sharp. As shown in Fig. \ref{tu-2},
$G$ is a minimal bicritical graph with exactly four cubic vertices.

\begin{figure}[h]
\centering
\includegraphics[height=2cm,width=8cm]{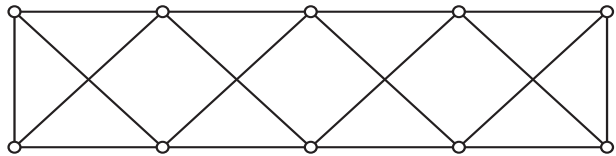}
\caption{\label{tu-2} A minimal bicritical graph $G$ with exactly four cubic vertices.}
\end{figure}

\bigskip
\noindent {\bf Acknowledgements}

The authors would like to thank the referees for their valuable comments,
which helped us to improve the presentation of the paper.

\end{document}